\documentclass[12pt]{amsart}

\usepackage{amsmath, amsthm}
\usepackage{graphicx,epsfig}
\usepackage{amsmath,latexsym,amssymb,verbatim}

\newtheorem{definition}{Definition}

\newtheorem{proposition}[definition]{Proposition}
\newtheorem{theorem}[definition]{Theorem}
\newtheorem{corollary}[definition]{Corollary}
\newtheorem{remark}{Remark}
\newtheorem{example}{Example}

\def\sideremark#1{\ifvmode\leavevmode\fi\vadjust{\vbox to0pt{\vss % the remark
      \hbox to 0pt{\hskip\hsize\hskip1em           %                will appear only
 \vbox{\hsize2cm\tiny\raggedright\pretolerance10000%                on the side
 \noindent #1\hfill}\hss}\vbox to8pt{\vfil}\vss}}}%
\def\meas{\operatorname{vol}}

\begin{document}

\title[Minkowski content, ambient space]{Invariance of the normalized Minkowski content with respect to the ambient space}
\address{University of Zagreb, Department of Applied Mathematics, Faculty of Electrical Engineering and Computing, Unska 3, 10000 Zagreb, Croatia} 
\email{maja.resman@fer.hr}
\author{Maja Resman}

\begin{abstract}
It is easy to show that the lower and the upper box dimensions of a bounded set in Euclidean space are invariant with respect to the ambient space. In this article we show that the Minkowski content of a Minkowski measurable set is also invariant with respect to the ambient space when normalized by an
 appropriate constant. In other words, the value of the normalized Minkowski content of a bounded, Minkowski measurable set is intrinsic to the set.
\end{abstract}

\maketitle

\section{Introduction}\label{sec1}

In the last century, there has been a growing interest for computing fractal dimensions of various sets. The notion of fractal dimension goes back to H.\ Minkowski (Minkowski dimension), H. Hausdorff (Hausdorff dimension) and G. Bouligand (Bouligand dimension). In dynamical systems, fractal dimensions of invariant sets were used to measure the complexity of systems. For a short overview of applications of fractal dimensions in dynamics see e.g. \cite{enc}. 

We deal here with box dimensions and Minkowski contents of sets. Among other applications, it has been noted that the box dimension and the Minkowski content of spiral trajectories or of trajectories of Poincar\' e maps around limit periodic sets show their cyclicity in perturbations, see e.g. \cite{marezu} and \cite{zuzu}. Also, box dimension and Minkowski content of a set are closely related to the domain of analyticity of the associated distance zeta function, see e.g. \cite{larazu}.

A notion closely related to Minkowski content is Minkowski measurability. We cite some articles dealing with Minkowski measurability, for example \cite{falcart}, \cite{lappom} and \cite{saunders}. In their study of the Weyl-Berry conjecture, Lapidus and Pomerance in \cite{lappom} characterized compact sets on the real line which are Minkowski measurable. They provided a method for constructing Minkowski measurable sets on the real line of any box dimension $d\in(0,1)$. They further showed that the Cantor set is not Minkowski measurable on the real line.  

Bilipschitz transformations preserve the box dimension of a set, see e.g. \cite{falconer}. On the other hand, not a lot about preserving Minkowski measurability is known. Bilipschitz mappings in general do not preserve Minkowski measurability of sets, see \cite{saunders}. Sufficiently general conditions imposed on mappings which ensure that Minkowski measurability is preserved have not yet been proposed. 

In this article, we pose the following questions: 
\emph{
\begin{enumerate}
\item[(i)] Is Minkowski measurability of a set in Euclidean space preserved when the set is embedded in a higher-dimensional Euclidean space?
\item[(ii)] Is Minkowski content of a set invariant with respect to the ambient Euclidean space in which we consider the set?
\end{enumerate}}\noindent Our results are stated in Theorem~\ref{main} and Theorem~\ref{NMC} in Section~\ref{sec3}.
\bigskip

Let us introduce the main notions. Let $U\subset \mathbb{R}^N$ be a bounded set. For $\varepsilon>0$, by $U_\varepsilon^N$ we denote its $\varepsilon$-neighborhood in $\mathbb{R}^N$:
$$
U_\varepsilon^N=\{x\in\mathbb{R^N}:d_N(x,A)\leq\varepsilon\},
$$
where $d_N$ denotes the Euclidean metric in $\mathbb{R}^N$. 

\medskip
For the following definitions, see e.g. \cite{tricot}. Let $\meas_N{(U_\varepsilon^N)}$ denote the Lebesgue measure of $U_\varepsilon^N$ in $\mathbb{R}^N$.

The \emph{lower $s$-dimensional Minkowski content} of a bounded set $U\subset\mathbb{R}^N$, $s\in[0,N]$, is defined as the limit
\begin{equation}\label{mc}
\mathcal{M}^s_*(U,\mathbb{R}^N)=\liminf_{\varepsilon\to 0}\frac{\meas_N{(U_\varepsilon^N)}}{\varepsilon^{N-s}}\in[0,\infty].
\end{equation}
Similarly, we define \emph{the upper $s$-dimensional Minkowski content} $\mathcal{M}^{*s}(U,\mathbb{R}^N)$, using $\limsup$ instead of $\liminf$. If the upper and the lower Minkowski contents agree, then the $s$-\emph{dimensional Minkowski content of U}, denoted by $\mathcal{M}^s(U,\mathbb{R}^N)$, is defined to be this common value.

Furthermore, \emph{the lower and upper box dimension} of the set $U\subset\mathbb{R}^N$ are defined respectively as
\begin{align*}
\underline{\dim}_B(U,\mathbb{R}^N)=&\sup\{s\geq0:\mathcal{M}_{*}^s(U,\mathbb{R}^N)=\infty\}\\
=&\inf\{s\geq0:\mathcal{M}_{*}^s(U,\mathbb{R}^N)=0\},\\
\overline{\dim}_B(U,\mathbb{R}^N)=&\sup\{s\geq0:\mathcal{M}^{*s}(U,\mathbb{R}^N)=\infty\}\\
=&\inf\{s\geq 0:\mathcal{M}^{*s}(U,\mathbb{R}^N)=0\}.
\end{align*}
They can be described as the moments of jump of the Minkowski contents $\mathcal{M}_*^s$ and $\mathcal{M}^{*s}$ respectively from $\infty$ to $0$, as $s$ grows in $[0,N]$. In the literature, the upper box dimension is sometimes called the \emph{limit capacity}, see \cite{PT}.

If $\underline{\dim}_B(U,\mathbb{R}^N)=\overline{\dim}_B(U,\mathbb{R}^N)$, we define \emph{the box dimension of the set $U\subset\mathbb{R}^N$} as the common value:
$$
\dim_B(U,\mathbb{R}^N)=\underline{\dim}_{B}(U,\mathbb{R}^N)=\overline{\dim}_{B}(U,\mathbb{R}^N).
$$

\medskip

Let us suppose now that the set $U\subset{\mathbb{R}^N}$ has box dimension $d=\dim_B(U,\mathbb{R}^N)$. If the upper and the lower $d$-dimensional Minkowski contents of $U\subset\mathbb{R}^N$ are both in $(0,\infty)$, we say that the set $U$ is \emph{Minkowski nondegenerate} in $\mathbb{R}^N$. If moreover both Minkowski contents agree, that is, if $\mathcal{M}^{*d}(U,\mathbb{R}^N)=\mathcal{M}_*^d(U,\mathbb{R}^N)\in(0,\infty)$, we say that the set $U$ is \emph{Minkowski measurable} in $\mathbb{R}^N$. For simplicity, in this case, the $d$-dimensional Minkowski content is called the Minkowski content and denoted simply by $\mathcal{M}(U,\mathbb{R}^N)$.

\section{Motivation}\label{sec2}
We state the result about the invariance of box dimension with respect to the ambient space in which we consider the set. That is, the box dimension of a set is an intrinsic property of the set. We were not able to find this result explicitely stated in the literature.
\begin{proposition}\label{invdim}
Let $U\subset\mathbb{R}^N\hookrightarrow\mathbb{R}^{N+1}$ be a bounded set. Then it holds that
\begin{align*}
\underline{\dim}_B(U,\mathbb{R}^N)&=\underline{\dim}_B(U,\mathbb{R}^{N+1}),\\
\overline{\dim}_B(U,\mathbb{R}^N)&=\overline{\dim}_B(U,\mathbb{R}^{N+1}),\\
\end{align*}
where $\dim_B(U,\mathbb{R}^N)$ denotes the box dimension of $U$ when regarded as a subset of $\mathbb{R}^N$.
\end{proposition}

\begin{proof} Let $f:\mathbb{R}^N\to\mathbb{R}^{N+1}$, $f(\mathbf{x})=(\mathbf{x},0)$. Then $f$ is obviously a bilipschitz mapping from $U\subset\mathbb{R}^N$ onto $U=U\times\{0\} \subset\mathbb{R}^{N+1}$. By \cite{falconer}, Section 3.2, the upper and the lower box dimensions are preserved under bilipschitz mappings. 
\end{proof}
\smallskip
By Proposition \ref{invdim}, we can denote the lower and the upper box dimensions of the set $U\subset\mathbb{R}^N$ by $\underline{\dim}_B(U)$ and $\overline{\dim}_B(U)$ respectively, without mentioning the ambient space where we consider the set.
\smallskip

Our goal is to obtain a similar result for the Minkowski content of a set $U$. We cannot proceed as in the above proof, since the Minkowski content is not invariant under the bilipschitz mappings. Moreover, the property of Minkowski measurability is not preserved even by $C^1$ bilipschitz mappings, see \cite{saunders}.

A bounded set $U\subset\mathbb{R}^N$ can be embedded in $\mathbb{R}^{N+1}$ as the Cartesian product $U\times\{0\}$. For the Minkowski content of the Cartesian product, the following estimates hold, see \cite[Theorem 3.3.6]{krantz}:
\begin{proposition}\label{kr}
If $A\subset\mathbb{R}^M$ and $B\subset\mathbb{R}^N$, then the following result for the Minkowski contents of the Cartesian product $A\times B\subset \mathbb{R}^M\times\mathbb{R}^N$ holds:
\begin{align*}
\sqrt{2}&^{-(M+N-s-r)/2}\cdot\mathcal{M}_*^s(A,\mathbb{R}^M)\cdot \mathcal{M}_*^r(B,\mathbb{R}^N)\leq\\
&\qquad  \leq\mathcal{M}_*^{r+s}(A\times B,\mathbb{R}^{N+M})\leq \mathcal{M}^{*(r+s)}(A\times B,\mathbb{R}^{N+M})\leq\\
&\hspace{6.5cm} \leq\mathcal{M}^{*s}(A,\mathbb{R}^M)\cdot \mathcal{M}^{*r}(B,\mathbb{R}^N),
\end{align*}
where $0\leq s\leq M$, $0\leq r\leq N$.
\end{proposition}

Applying the above proposition to the set $U\subset\mathbb{R}^N\hookrightarrow \mathbb{R}^N$, we get the inequalities involving the ambient spaces:
\begin{equation}\label{ineq}2^{-\frac{N-1-s}{2}}\mathcal{M}_*^s(U,\mathbb{R}^N)\leq \mathcal{M}_*^s(U,\mathbb{R}^{N+1})\leq \mathcal{M}^{*s}(U,\mathbb{R}^{N+1})\leq 2\mathcal{M}^{*s}(U,\mathbb{R}^N),\end{equation}
where $0\leq s\leq N$.

From inequality \eqref{ineq}, nothing can be said about preserving the $s$-dimensional Minkowski content of a bounded set in $\mathbb{R}^N$, when embedding it into $\mathbb{R}^{N+1}$. In Theorem~\ref{main} in Section~\ref{sec3}, we improve the constants from \eqref{ineq} by making them equal. With improved constants, in Theorem~\ref{NMC}, we show that the Minkowski content of a Minkowski measurable set is invariant with respect to the ambient space, when multiplied by an appropriate constant.

\smallskip

Finally, the next paragraph gives us an idea about the constant by which we should normalize the Minkowski content to ensure invariance.
In e.g. \cite{mattila}, the Minkowski contents are defined in the same way as in Section~\ref{sec1}, that is, without any normalizing constant. Let us recall an alternative definition of Minkowski contents from e.g. \cite{federer} and \cite{krantz}. The $s$-dimensional Minkowski contents defined in \eqref{mc} are additionally divided by the constant
\begin{equation}\label{gamma}
\gamma_{N-s}=\frac{\pi^{\frac{N-s}{2}}}{\Gamma(\frac{N-s}{2}+1)}.
\end{equation}
We will call them \emph{the normalized lower and upper $s$-dimensional Minkowski contents},
\begin{equation}\label{norma}
\mathcal{\overline{M}}_*^s(U,\mathbb{R}^N)=\frac{\mathcal{M}_*^s(U,\mathbb{R}^N)}{\gamma_{N-s}},\ \mathcal{\overline{M}}^{*s}(U,\mathbb{R}^N)=\frac{\mathcal{M}^{*s}(U,\mathbb{R}^N)}{\gamma_{N-s}}.
\end{equation}
Here, $\Gamma$ is the usual gamma function. For any integer $s\in[0,N)$, $\gamma_{N-s}$ is equal to the volume of the unit ball in $\mathbb{R}^{N-s}$. 

As before, if $\mathcal{\overline{M}}_*^s(U,\mathbb{R}^N)=\mathcal{\overline{M}}^{*s}(U,\mathbb{R}^N)$, the common value will be called \emph{the normalized $s$-dimensional Minkowski content} and denoted by $\overline{\mathcal{M}}^s(U,\mathbb{R}^N)$. Furthermore, if $s=\dim_B(U)$, we omit writing the superscript $d$ and write only $\mathcal{\overline{M}}_*(U,\mathbb{R}^N)$, $\mathcal{\overline{M}}^{*}(U,\mathbb{R}^N)$ and $\mathcal{\overline{M}}(U,\mathbb{R}^N)$.

This normalization ensures that, for an integer $k$, $1\leq k\leq N$, the normalized $k$-dimensional Minkowski content of a $k$-rectifiable set in $\mathbb{R}^N$ coincides with its $k$-dimensional Hausdorff measure, which is an intrinsic value of a set, equal to the $k$-dimensional Lebesgue measure of the set, within the constant multiple depending only on $k$. It is therefore independent of the ambient space $\mathbb{R}^N$, for $N\geq k$. We cite \cite[Theorem 3.3.4]{krantz}:
\begin{proposition}\label{Hauss} 
Suppose $1 \leq k \leq N$ is an integer. 
Let\ $U\subset\mathbb{R}^N$ be closed and let $U\subset f(\mathbb{R}^k)$, for some Lipschitz function $f: \mathbb{R}^k\to\mathbb{R}^N$. Then
the $k$-dimensional Minkowski content of $U$ exists and its normalization is equal to the $k$-dimensional Hausdorff measure $\mathcal{H}^k(U)$ of the set, that is,
$$
\mathcal{\overline{M}}^{k}(U,\mathbb{R}^N)=\mathcal{H}^{k}(U).
$$
\end{proposition}

It remains however the question of invariance of the normalized Minkowski content of an arbitrary Minkowski measurable set of box dimension $d$ with respect to the ambient space, also when $d\in[0,N)$ is noninteger.
We will show in Section~\ref{sec3} that the answer to this question is positive, even for noninteger $d$.

\section{Results}\label{sec3}
In this section, we state our main results. Let $\gamma_{N-s}$ be as in \eqref{gamma}.
\smallskip

The main result is the following theorem. It states that the Minkowski measurability is preserved when embedding the set in higher-dimensional space. Thus its normalized Minkowski content, as defined in e.g. \cite{federer} and \cite{krantz}, becomes independent of the ambient space.
\begin{theorem}\label{NMC}
Let $U\subset\mathbb{R}^N$ be a Minkowski measurable set in $\mathbb{R}^N$, with the box dimension $d\in[0,N]$. Then $U$ is also Minkowski measurable regarded as a subset of $\mathbb{R}^{N+1}$. Moreover, for Minkowski contents in $\mathbb{R}^{N}$ and $\mathbb{R}^{N+1}$, it holds that
\begin{equation}\label{incl}
\frac{\mathcal{M}(U,\mathbb{R}^{N+1})}{\gamma_{N+1-d}}=\frac{\mathcal{M}(U,\mathbb{R}^N)}{\gamma_{N-d}}. 
\end{equation}
In other words, the normalized Minkowski content of a Minkowski measurable set from $\mathbb{R}^N$ remains unchanged by the embedding $\mathbb{R}^N\hookrightarrow \mathbb{R}^{N+1}$,
\begin{equation}\label{includ}
\overline{\mathcal{M}}(U,\mathbb{R}^{N+1})=\overline{\mathcal{M}}(U,\mathbb{R}^N).
\end{equation}
\end{theorem}
\begin{proof}
The proof follows directly from Theorem~\ref{main} below. Since the set $U$ is Minkowski measurable, all the inequalities in \eqref{compare} become equalities.
\end{proof}

The following Theorem gives the inequalities concerning the upper and the lower Minkowski contents in ambient spaces, in general case, also when $U$ is not Minkowski measurable.
\begin{theorem}[Minkowski contents and embedding]\label{main}
Let $U\subset\mathbb{R}^N$ be a bounded set and let $0\leq s\leq N$. Then the following inequalities between $s$-dimensional Minkowski contents in ambient spaces $\mathbb{R}^N$ and $\mathbb{R}^{N+1}$ hold:
\begin{align}\label{compare}
\frac{\mathcal{M}^s_*(U,\mathbb{R}^N)}{\gamma_{N-s}}&\leq \frac{\mathcal{M}^s_*(U,\mathbb{R}^{N+1})}{\gamma_{N+1-s}}\leq\\
&\qquad\leq \frac{\mathcal{M}^{*s}(U,\mathbb{R}^{N+1})}{\gamma_{N+1-s}}\leq\frac{\mathcal{M}^{*s}(U,\mathbb{R}^{N})}{\gamma_{N-s}}.\nonumber
\end{align}
The above inequalities hold also in the case of Minkowski contents equal to $0$ or $\infty$.
\end{theorem}

\begin{remark}[Optimality of constants in \eqref{compare}]\label{opt} Let us reformulate \eqref{compare}.
For any $s\in[0,N]$ and for any bounded set $U\subset\mathbb{R}^N$, it holds that
\begin{align}
\mathcal{M}^s_*(U,\mathbb{R}^{N+1})&\geq \mathcal{M}^s_*(U,\mathbb{R}^N)\frac{\gamma_{N+1-s}}{\gamma_{N-s}},\label{compare1}\\ \mathcal{M}^{*s}(U,\mathbb{R}^{N+1})&\leq\mathcal{M}^{*s}(U,\mathbb{R}^{N})\frac{\gamma_{N+1-s}}{\gamma_{N-s}}.\nonumber
\end{align}
The constant $\frac{\gamma_{N+1-s}}{\gamma_{N-s}}$ in $\eqref{compare1}$ is optimal for a given $s\in[0,N]$. Let $s\in\mathbb[0,N]$. By Theorem~3 in \cite{saunders}, there exists a Minkowski measurable set in $\mathbb{R}^N$ with box dimension equal to $s$. By Theorem~\ref{NMC} above, for Minkowski measurable sets, inequalities in \eqref{compare1} become equalities. This proves that the constant $\frac{\gamma_{N+1-s}}{\gamma_{N-s}}$ in $\eqref{compare1}$ is the best possible.

\smallskip
Let us comment here on an alternative proof of the fact used above: for every $s\in[0,N]$, there exists a Minkowski measurable set $U\subset\mathbb{R}^N$, such that $\dim_B U=s$. First, on the real line, one can construct a Minkowski measurable set of any box dimension $d\in [0,1]$. The set can be constructed using fractal strings, as in \cite{lappom}, or as a discrete orbit generated by function $g(x)=x-x^\alpha$, $\alpha\in\mathbb{R},\ \alpha>1$, as in \cite{neveda}. It is easy to prove that, if $U\subset\mathbb{R}^N$ is Minkowski measurable in $\mathbb{R}^N$, with box dimension $d$, then $U\times[0,1]$ is Minkowski measurable in $\mathbb{R}^{N+1}$, with box dimension $d+1$. Moreover, the values of their Minkowski contents (in $\mathbb{R}^N$ and $\mathbb{R}^{N+1}$ respectively) are the same. The proof follows directly from definition of Minkowski content and the obvious geometric fact: $$\meas_{N+1}\big((U\times[0,1]\big)_\varepsilon^{\ N+1})=\meas_{N}(U_\varepsilon^N)\cdot 1+\meas_{N+1}(U_\varepsilon^{N+1}).$$
\end{remark}
In the proof of Theorem~\ref{main}, we use the following two auxiliary propositions.
\begin{proposition}\label{enhood}
Let $U\subset\mathbb{R}^N$ be a bounded set. Let $U_\varepsilon^N$ and $U_\varepsilon^{N+1}$ denote the $\varepsilon$-neighborhoods of $U$ in $\mathbb{R}^N$ and $\mathbb{R}^{N+1}$ respectively. For $\varepsilon>0$, it holds that
$$
\meas_{N+1}{(U_\varepsilon^{N+1})}=2\int_{0}^{\varepsilon}\meas_{N}{(U_{\sqrt{\varepsilon^2-y^2}}^N)}dy,
$$
where $\meas_{N}$ denotes the Lebesgue measure in $\mathbb{R}^{N}$.
\end{proposition}

\begin{proof}
Let $U\subset\mathbb{R}^N\hookrightarrow \mathbb{R}^{N+1}$. We fix the coordinate system in $\mathbb{R}^{N+1}$ such that $U$ lies in the $N$-dimensional plane $\{\mathbf{x}=0\}$. Let $d_N$ and $d_{N+1}$ denote the Euclidean distances in $\mathbb{R}^N$ and $\mathbb{R}^{N+1}$ respectively.\\
We have
\begin{align}\label{comp1}
\meas_{N+1}{(U_\varepsilon^{N+1})}=&\int\int\ldots \int_{\{(\mathbf{x},y)\in\mathbb{R}^N\times\mathbb{R}\ :\ d_{N+1}((\mathbf{x},y),\ U)\leq \varepsilon\}} 1\cdot d\mathbf{x}\ dy.
\end{align}
Obviously, $d_{N+1}((\mathbf{x},y),U)=\sqrt{d_N(\mathbf{x},U)^2+y^2}$. For a fixed $y\in[-\varepsilon,\varepsilon]$, the following sets are equal:
\begin{equation}\label{comp2}\{\mathbf{x}\in\mathbb{R}^N:d_{N+1}((\mathbf{x},y),\ U)\leq \varepsilon\}=\{\mathbf{x}\in\mathbb{R}^N: d_{N}(\mathbf{x},U)\leq \sqrt{\varepsilon^2-y^2}\}.\end{equation} Using Fubini's theorem, from \eqref{comp1} and \eqref{comp2} we get
\begin{align*}
\meas_{N+1}{(U_\varepsilon^{N+1})}=&\int_{-\varepsilon}^{\varepsilon}dy \int_{\{\mathbf{x}\in\mathbb{R}^N\ :\ d_N(\mathbf{x},U)\leq \sqrt{\varepsilon^2-y^2}\}}1\cdot d\mathbf{x}\\
=&\int_{-\varepsilon}^{\varepsilon}\meas_{N}{(U_{\sqrt{\varepsilon^2-y^2}}^N)}\ dy\\
=&2\int_0^{\varepsilon}\meas_{N}{(U_{\sqrt{\varepsilon^2-y^2}}^N)}\ dy.
\end{align*}
\end{proof}

\begin{proposition}\label{ball}
Let $0\leq s\leq N$ and $\varepsilon>0$. It holds that
$$
2\int_0^\varepsilon (\sqrt{\varepsilon^2-y^2})^{N-s}\ dy=\frac{\gamma_{N+1-s}}{\gamma_{N-s}}\cdot \varepsilon^{N+1-s},
$$
where $\gamma_{N-s}$ is defined in \eqref{gamma}.
\end{proposition}
\begin{proof} The above integral is computed substituting $y=\varepsilon\sin t$. We get
\begin{align*}
2\int_0^\varepsilon (\sqrt{\varepsilon^2-y^2})^{N-s}\ dy&=\varepsilon^{N+1-s}\cdot 2\int_0^{\frac{\pi}{2}}(\cos t)^{N+1-s}\ dt\\
&=\varepsilon^{N+1-s}\cdot\mathcal{B}(\frac{1}{2},\frac{N-s}{2}+1)\\
&=\frac{\sqrt{\pi}\ \Gamma(\frac{N-s}{2}+1)}{\Gamma(\frac{N-s}{2}+\frac{3}{2})}\cdot \varepsilon^{N+1-s}\\
&=\frac{\gamma_{N+1-s}}{\gamma_{N-s}}\cdot\varepsilon^{N+1-s}.
\end{align*}
Here, $\mathcal{B}(x,y)$ denotes the Beta function and the equalities follow from its relation with the Gamma function, which can be found in any book on special functions, see e.g. \cite{moll}.
\end{proof}

\medskip
\noindent \emph{Proof of Theorem~\ref{main}}.

Let us prove the first inequality in \eqref{compare}. That is,
$$
\mathcal{M}^s_*(U,\mathbb{R}^N)\cdot\frac{\gamma_{N+1-s}}{\gamma_{N-s}}\leq\mathcal{M}^s_*(U,\mathbb{R}^{N+1}).
$$
Suppose first that $\mathcal{M}^s_*(U,\mathbb{R}^{N})\in (0,\infty)$. From
$$
\mathcal{M}^s_*(U,\mathbb{R}^{N})=\liminf_{\varepsilon\to 0}\frac{\meas_{N}{(U_\varepsilon^{N})}}{\varepsilon^{N-s}},
$$
by the definition of the limit inferior, we get that for each $\delta>0$, there exists $\varepsilon_\delta>0$, such that for all $ \varepsilon\leq \varepsilon_\delta$,
\begin{equation}\label{aux}
\meas_{N}{(U_\varepsilon^{N})}\geq (\mathcal{M}^s_*(U,\mathbb{R}^{N})-\delta)\cdot \varepsilon^{N-s}.
\end{equation}
By Proposition~\ref{enhood}, substituting \eqref{aux} in the integral, we get that for each $\delta>0$, there exists $\varepsilon_\delta>0$, such that for all $\varepsilon<\varepsilon_\delta$,
\begin{align}
\meas_{N+1}{(U_\varepsilon^{N+1})}&\geq 2(\mathcal{M}_*^s(U,\mathbb{R}^N)-\delta)\int_0^\varepsilon\sqrt{\varepsilon^2-y^2}^{N-s} dy,\nonumber\\
&\geq (\mathcal{M}_*^s(U,\mathbb{R}^N)-\delta)\varepsilon^{N+1-s}\frac{\gamma_{N+1-s}}{\gamma_{N-s}}.\label{u}
\end{align}
Here, the last inequality is obtained using Proposition~\ref{ball}.

Reformulating \eqref{u}, for each $\delta>0$, there exists $\varepsilon_{\delta}$, such that for all $\varepsilon<\varepsilon_{\delta}$,
\begin{equation}\label{conc}
\frac{\meas_{N+1}{(U_\varepsilon^{N+1})}}{\varepsilon^{N+1-s}}\geq \mathcal{M}_*^s(U,\mathbb{R}^{N})\cdot\frac{\gamma_{N+1-s}}{\gamma_{N-s}}-\delta.
\end{equation}
Since $\mathcal{M}^s_*(U,\mathbb{R}^{N+1})=\liminf_{\varepsilon\to 0}\frac{\meas_{N+1}{(U_\varepsilon^{N+1})}}{\varepsilon^{N+1-s}}$, using \eqref{conc}, we conclude that 
\begin{equation}\label{mink}
\mathcal{M}^s_*(U,\mathbb{R}^{N+1})\geq\mathcal{M}_*^s(U,\mathbb{R}^N)\frac{\gamma_{N+1-s}}{\gamma_{N-s}} .
\end{equation} Note that from \eqref{conc} it is not possible to conclude equality in \eqref{mink}.

In the case when $\mathcal{M}_*^s(U,\mathbb{R}^N)=\infty$, the proof is similar. The case when $\mathcal{M}_*^s(U,\mathbb{R}^N)=0$ is obvious.
 
The last inequality in \eqref{compare},
$$
\mathcal{M}^{*s}(U,\mathbb{R}^{N+1})\leq\mathcal{M}^{*s}(U,\mathbb{R}^{N})\cdot\frac{\gamma_{N+1-s}}{\gamma_{N-s}},
$$ 
can be proven analogously, using $\limsup$ instead of $\liminf$.
\qed
\bigskip

From Theorem~\ref{main} and Remark~\ref{opt}, we immediately derive the following interesting consequence.
\begin{corollary}\label{fam}
Let $s\in[0,N]$. For a given $s$, let $\mathcal{U}_*^s$ and $\mathcal{U}^{*s}$ denote the following families of sets:
\begin{align*}
\mathcal{U}_*^s&=\{U\subset\mathbb{R}^N : \mathcal{M}_*^s(U,\mathbb{R}^N)\in(0,\infty)\},\\
\mathcal{U}^{*s}&=\{U\subset\mathbb{R}^N : \mathcal{M}^{*s}(U,\mathbb{R}^N)\in(0,\infty)\},
\end{align*} 
which are nonempty by Remark~\ref{opt}.
It holds that
$$
\min_{U\in\mathcal{U}_*^s}\frac{\mathcal{M}_*^s(U,\mathbb{R}^{N+1})}{\mathcal{M}_*^s(U,\mathbb{R}^N)}=\max_{U\in\mathcal{U}^{*s}}\frac{\mathcal{M}^{*s}(U,\mathbb{R}^{N+1})}{\mathcal{M}^{*s}(U,\mathbb{R}^N)}=\frac{\gamma_{N+1-s}}{\gamma_{N-s}}.
$$
Equivalently, for normalized Minkowski contents defined in \eqref{norma}, it holds that
$$
\min_{U\in\mathcal{U}_*^s}\frac{\overline{\mathcal{M}}_*^s(U,\mathbb{R}^{N+1})}{\overline{\mathcal{M}}_*^s(U,\mathbb{R}^N)}=\max_{U\in\mathcal{U}^{*s}}\frac{\overline{\mathcal{M}}^{*s}(U,\mathbb{R}^{N+1})}{\overline{\mathcal{M}}^{*s}(U,\mathbb{R}^N)}=1.
$$
Moreover, the minimum and the maximum are achieved for all Minkowski measurable sets of box dimension $s$.
\end{corollary} 

\bigskip

We conclude this section with some remarks.

\begin{remark}[The converse]
It remains open if the converse of the statement in Theorem~\ref{NMC} holds. That is, assume that $U\subset\mathbb{R}^N$. If $U$ is Minkowski measurable in $\mathbb{R}^{N+1}$, does it follow that $U$ is also Minkowski measurable in the lower-dimensional space $\mathbb{R}^N$?
\end{remark}

\begin{remark}[Invariance of upper and lower Minkowski content]
The question that remains open is if the invariance with respect to the ambient space \eqref{includ} holds separately for normalized upper and normalized lower Minkowski contents. That is, if the first and the last inequalities in \eqref{compare} are in fact equalities for any $U\in\mathbb{R}^N$, not necessarily Minkowski measurable. As already stated, equalities cannot be obtained using the proof proposed above. The question remains if we can obtain equalities in some other way or, on the contrary, construct a counterexample to show that
the answer is negative. A counterexample here would be any bounded set in $\mathbb{R}^N$, which is not Minkowski measurable, and whose normalized lower or upper Minkowski content are not invariant with respect to the ambient space, if such exists.

\end{remark}
\medskip

To illustrate the results of Theorem~\ref{NMC}, let us compute the Minkowski contents of some basic sets, when regarded in the ambient spaces of different dimensions.
\begin{example}\

\begin{enumerate}
\item Let $U=\{x\}\subset\mathbb{R}$. Directly using definitions, we compute:
$$
\dim_B U=0,\ \mathcal{M}(U,\mathbb{R})=2,\ \mathcal{M}(U,\mathbb{R}^2)=\pi.
$$
On the other hand, by \eqref{gamma}, $\gamma_2/\gamma_1=\pi/2$, so the formula \eqref{incl} holds.
\medskip

\item Let $U=[a,b]$ be a segment of length $l=b-a$ in $\mathbb{R}$. Then
$$
\dim_B U=1,\ \mathcal{M}(U,\mathbb{R})=l,\ \mathcal{M}(U,\mathbb{R}^2)=2l.
$$
On the other hand, by \eqref{gamma}, $\gamma_1/\gamma_0=2/1=2$, and the formula \eqref{incl} holds.

\end{enumerate}
\end{example}
\medskip

\noindent \textbf{Acknowledgments.\ \ }\emph {I would like to thank Darko \v Zubrini\' c, who proposed the question of invariance of the Minkowski content with respect to the ambient space, and to Vesna \v Zupanovi\' c, for her useful suggestions.}

\end{document}